\RequirePackage{lineno}
\documentclass[12pt,a4paper,leqno]{amsart}
\newcounter{minutes}
\divide\time by 60
\newcounter{hours}
\multiply\time by 60
\addtocounter{minutes}{-\time}
 \def\registered{
 {\ooalign{\hfil\raise .00ex\hbox{\scriptsize R}\hfil\crcr\mathhexbox20D}}}

\usepackage{amssymb}
\usepackage{graphicx}
\date{}
\newfont{\cyrilic}{wncyr10 scaled 1000}
\newpage
\title{On generalized complete elliptic integrals and modular functions }
\author[B. A. Bhayo]{B. A. Bhayo}
\address{Department of Mathematics, University of Turku,
FI-20014 Turku, Finland} \email{barbha@utu.fi}

\author[M. Vuorinen]{M. Vuorinen}
\address{Department of Mathematics, University of Turku,
FI-20014 Turku, Finland} \email{vuorinen@utu.fi}
\newcommand{\comment}[1]{}

\swapnumbers
\theoremstyle{plain}
\newtheorem{theorem}[equation]{Theorem}
\newtheorem{lemma}[equation]{Lemma}

\newtheorem{corollary}[equation]{Corollary}

\newtheorem{remark}[equation]{Remark}

\newtheorem{subsec}[equation]{}

\numberwithin{equation}{section}

\begin{document}
\font\fFt=eusm10 
\font\fFa=eusm7  
\font\fFp=eusm5  
\def\K{\mathchoice
{\hbox{\,\fFt K}}
{\hbox{\,\fFt K}}
{\hbox{\,\fFa K}}
{\hbox{\,\fFp K}}}
\def\E{\mathchoice
{\hbox{\,\fFt E}}
{\hbox{\,\fFt E}}
{\hbox{\,\fFa E}}
{\hbox{\,\fFp E}}}

\maketitle
\vspace{.5cm}

\begin{abstract}
This paper deals with generalized elliptic integrals and
generalized modular functions. Several new inequalities are given for these
and related functions.
\end{abstract}



\def\thefootnote{}
\footnotetext{ \texttt{\tiny File:~\jobname .tex,
          printed: \number\year-\number\month-\number\day,
          \thehours.\ifnum\theminutes<10{0}\fi\theminutes}
} \makeatletter\def\thefootnote{\@arabic\c@footnote}\makeatother

\vspace{.5cm}

\section{Introduction}

During the past fifteen years, after the publication of the landmark paper
\cite{bbg}, numerous papers have been written about generalized elliptic
 integrals, modular functions and inequalities for them.
See e.g. \cite{aqvv, aq, b, b2, hlvv, hvv, wzc, wzqc, zwc1,zwc2}.
Modular equations have a long history, which goes back to the works of
A.M. Legendre, K.F. Gauss, C. Jacobi and S. Ramanujan about number theory.
Modular equations also occur in geometric function theory as shown in
 \cite{aqvv,vu,k,lv} and in numerical computations of moduli of
 quadrilaterals \cite{hrv}. For recent surveys of this topic from the point
 of view of geometric function
theory, see \cite{avv4,avv5,av}.
The study of these functions is motivated by potential applications to
geometric function theory and to number theory.

Given complex numbers $a,b$ and $c$ with $c\neq0,-1,-2,\ldots$,
the \emph{Gaussian hypergeometric function} is the
analytic continuation to the slit place $\mathbf{C}\setminus[1,\infty)$ of the series
$$F(a,b;c;z)={}_2F_1(a,b;c;z)=\sum^\infty_{n=0}\frac{(a,n)(b,n)}
{(c,n)}\frac{z^n}{n!},\qquad |z|<1.$$
Here $(a,0)=1$ for $a\neq 0$, and $(a,n)$ is the \emph{shifted factorial function}
 or the \emph{Appell symbol}
$$(a,n)=a(a+1)(a+2)\cdots(a+n-1)$$
for $n\in\mathbb{Z}_+$.

For later use we define classical \emph{gamma function} $\Gamma(x)$,
 and  \emph{beta function} $B(x,y)$.
For ${\rm Re}\, x>0$, ${\rm Re}\, y>0$, these functions are defined by
$$\Gamma(x)=\int^\infty_0 e^{-t}t^{x-1}\,dt,\,\,
B(x,y)=\frac{\Gamma(x)\Gamma(y)}{\Gamma(x+y)},$$
respectively.

For the formulation of our main results and for
later use we introduce some
basic notation.
The decreasing homeomorphism
$\mu_a:(0,1)\to (0,\infty)$ is defined by
$$\mu_a(r)=\frac{\pi}{2\,\sin(\pi\, a)}\frac{F(a,1-a;1;r^{'2})}
{F(a,1-a;1;r^2)}=
\frac{\pi}{2\,\sin(\pi\, a)}\frac{\K_a(r^{'})}{\K_a(r)}\,$$
for $r\in(0,1)$ and $r^{'}=\sqrt{1-r^2}$.
 A generalized modular equation with signature $1/a$ and
order (or degree) $p$ is
\begin{equation}\label{meq}
\mu_a(s)=p\,\mu_a(r),\quad 0<r<1.
\end{equation}
We denote
\begin{equation}\label{1a}
s=\varphi^a_K(r)\equiv\mu^{-1}_a(\mu_a(r)/K),\quad K\in(0,\infty), \, p= 1/K\,,
\end{equation}
which is the solution of (\ref{meq}).

For $a\in(0,1/2],\,K\in(0,\infty)
,\,r\in(0,1)$, we have by  \cite[Lemma 6.1]{aqvv}
\begin{equation}\label{1aa}
\varphi^a_K(r)^2+\varphi^a_{1/K}(r^{'})^2=1.
\end{equation}
For $a\in(0,1/2]\,,$ $r\in(0,1)$ and
$r^{'}=\sqrt{1-r^2}$, the generalized elliptic integrals are defined by

$$\left\{\begin{array}{lll} \K_a(r)=\displaystyle\frac{\pi}{2}\,F(a,1-a;1;r^2),\\
                          \K^{'}_a(r)=\K_a(r^{'}),\\
                      \K_a(0)=\displaystyle\frac{\pi}{2},\, \K_a(1)=\infty,
                     \end{array}\right.
\quad{\rm and}\quad \left\{\begin{array}{lll} \E_a(r)=\displaystyle\frac{\pi}{2}\,
                          F(a-1,1-a;1;r^2),\\
                          \E^{'}_a(r)=\E_a(r^{'}),\\
                      \E_a(0)=\displaystyle\frac{\pi}{2},\,
                      \E_a(1)=\displaystyle\frac
                      {\sin(\pi\,a)}{2(1-a)}.
                     \end{array}\right.$$
In this paper we study the modular function $\varphi^a_K(r)$ for general
$a\in(0,\frac{1}{2}]$, as well as
 related functions $\mu_a,\,\K_a,\, \eta^a_K$,\,
$\lambda_a$,  and their dependency on $r$ and $K$, where

$$\eta^a_K(x)=\left(\frac{s}{s^{'}}\right)^2,\, s=\varphi^a_K(r),\,
r=\sqrt{\frac{x}{1+x}}\,,\,
\text{ for }x,K\in(0,\infty),$$
and
\begin{equation}\label{1b}
\lambda_a(K)=\left(\frac{\varphi^a_K(1/\sqrt{2})}
{\varphi^a_{1/K}(1/\sqrt{2})}\right)^2
=\left(\frac{\mu^{-1}_a(\pi /(2K \sin(\pi\,a))}
{\mu^{-1}_a(\pi K/(2 \sin(\pi\,a))}\right)^2
=\eta^a_K(1).
\end{equation}
Motivated by \cite{lp} and \cite{bv} we define
for $p>1$ and $r\in(0,1)$,
$${\rm artanh}_p(x)=\int^x_0(1-t^p)^{-1}dt=xF\left
(1\,,\frac{1}{p};1+\frac{1}{p};x^p\right)\,.$$
Then ${\rm artanh}_2(x)$ is the usual inverse hyperbolic tangent {(\rm artanh)} function.

We give next some of the main results of this paper.
\begin{theorem}\label{nn2} For $a,b,c>0,$ and $r\in(0,1)$,
 the function $g(p)=F(a,b;c;r^{p})^{1/p}\,$
is decreasing in $p\in(0,\infty)$. In particular, for $p\geq 1$\\
$(1)\qquad\qquad F(a,b;c;r^{p})^{1/p}\leq F(a,b;c;r)
\leq F(a,b;c;r^{1/p})^{p}$\,,\\
$(2)\qquad\qquad\qquad \left(\displaystyle\frac{\pi}{2} \right)^{1-1/p}   \K_a(r^{p})^{1/p}\leq
\K_a(r)\leq \left(\displaystyle\frac{\pi}{2}\right)^{1-p}    \K_a(r^{1/p})^{p}$\,,\\
$(3)\qquad\qquad\qquad\qquad\left(\displaystyle\frac{\pi}{2}\right)^{1-p}
\E_a(r^{1/p})^p\leq
 \E_a(r)\leq \left(\displaystyle\frac{\pi}{2}\right)^{1-1/p}
 \E_a(r^{p})^{1/p}$.

\end{theorem}

H. Alzer and S.-L. Qiu have given the following bounds for
 $\K=\K_{1/2}$ in \cite[Theorem 18]{aq}
\begin{equation}\label{aqbn}
\frac{\pi}{2}\left(\frac{{\rm artanh}(r)}{r}\right)^{3/4}
<\K(r)<\frac{\pi}{2}\left(\frac{{\rm artanh}(r)}{r}\right)\,.
\end{equation}
In the following theorem we generalize their result to the case of $\K_a$, and
for the particular case $a=1/2$ our
upper bound is better than their bound in (\ref{aqbn}). For a graphical comparison of the
bounds see Figure 1 below.

\bigskip
\begin{theorem}\label{kk}
For $p\geq 2$ and $r\in(0,1)$, we have
\begin{eqnarray*}
\frac{\pi}{2}\left(\frac{{\rm artanh}_p(r)}{r}\right)^{1/2}
&<& \frac{\pi}{2}\left(1-\frac{p-1}{p^2}\log(1- r^2)\right)\\
&<&\K_a(r)<
\frac{\pi}{2}\left(1-\frac{2}{p\,\pi_p}\log (1-r^2)\right)\,,
\end{eqnarray*}
where $a=1/p$ and $\pi_p=2\pi/(p\sin(\pi/p))\,.$
\end{theorem}

In \cite[Theorem 5.6]{aqvv} (see also \cite[Theorem 1.5, 1.8]{bpv}) it was proved that for $a\in(0,1/2]$ we have

 $$\mu_a\left(\frac{r\,s}{1+r^{'}s^{'}}\right)\leq \mu_a(r)+\mu_a(s)\leq
 2\mu_a\left(\frac{\sqrt{2rs}}{\sqrt{1+r\,s+r^{'}\,s^{'}}}\right)\,,$$
 for all $r,s\in(0,1)$. This inequality will be generalized below in Theorem \ref{muac}.
 In the next theorem we give a similar result for
 the function $\K_a$.

\begin{theorem}\label{kka} The function $f(x)=1/\K_a(1/\cosh(x))$ is increasing
 and concave from $(0,\infty)$ onto $(0,2/\pi)$. In particular,
 $$\frac{\K_a(r)\K_a(s)}{\K_a(rs/(1+r^{'}s^{'}))}\leq \K_a(r)+\K_a(s)\leq
 \frac{2\K_a(r)\K_a(s)}{\K_a(\sqrt{rs/(1+rs+r^{'}s^{'})})}\leq
 \frac{2\K_a(r)\K_a(s)}{\K_a(rs)}\,,$$
 for all $r,s\in(0,1)$, with equality in the third inequality if and only if $r=s$.
\end{theorem}

There are several bounds for the function $\mu_a(r)$ when $a=1/2$ in \cite[Chap.5]{avvb}. In the next theorem
we give a twosided bound for $\mu_a(r)$.

\begin{theorem} \label{mymu} For $p\geq 2$ and $r\in(0,1)$,
let
$$l_p(r)=\left(\frac{\pi_p}{2}\right)^2\left(\frac{p^2-(p-1)\log r^2}
{p\,\pi_p-2\log r^{'2}}\right)\quad and \quad u_p(r)=\left(\frac{p}{2}\right)^2
\left(\frac{p\,\pi_p-2\log r^2}{p^2-(p-1)\log r^{'2}}\right)\,.$$
(1) The following inequalities hold
$$l_p(r)<\mu_a(r)<u_p(r)\,,$$
where $a=1/p$\,.\\
(2) For $p=2$ we have
$$u_2(r)< \frac{4}{\pi}\,l_2(r)\,.$$
\end{theorem}


\bigskip
{\sc Acknowledgments.} The first author is indebted to the Graduate
School of Mathematical Analysis and its Applications for support.
The second author was, in part, supported by the Academy of Finland,
Project 2600066611. Both authors wish to acknowledge the expert help of
Dr. H. Ruskeep\"a\"a in the use of the Mathematica$^{\tiny \textregistered}$ software \cite{ru}. The authors also acknowledge the constructive suggestions of the referee.

\section{Proofs of Theorems \ref{nn2},\ref{kk},\ref{kka} and \ref{mymu}. }

For easy reference we record the next two lemmas from \cite{avvb} which
have found many applications. Some of the applications are reviewed in \cite{avv5}.
The first result sometimes called the \emph{monotone l'Hospital rule}.
\begin{lemma}\label{pee}\cite[Theorem 1.25]{avvb}
For $-\infty<a<b<\infty$,
let $f,g:[a,b]\to \mathbb{R}$
be continuous on $[a,b]$, and be differentiable on
$(a,b)$. Let $g^{'}(x)\neq 0$
on $(a,b)$. If $f^{'}(x)/g^{'}(x)$ is increasing
(decreasing) on $(a,b)$, then so are
$$[f(x)-f(a)]/[g(x)-g(a)]\quad and \quad [f(x)-f(b)]/[g(x)-g(b)].$$
If $f^{'}(x)/g^{'}(x)$ is strictly monotone,
then the monotonicity in the conclusion
is also strict.
\end{lemma}

\begin{lemma}\label{pee1} \cite[Lemma 1.24]{avvb}
For $p\in(0,\infty]$, let $I=[0,p)$, and suppose that
$f,g:I\to[0,\infty)$ are functions such that $f(x)/g(x)$ is decreasing on $I\setminus\{0\}$
and $g(0)=0$ and $g(x)>0$ for $x>0$. Then
$$f(x+y)(g(x)+g(y))\leq g(x+y)(f(x)+f(y))\,,$$
for $x,y,x+y\in I$. Moreover, if the monotonicity of $f(x)/g(x)$ is strict then the above inequality is also
strict on $I\setminus\{0\}$.
\end{lemma}

For easy reference we recall the following lemmas from \cite{aqvv}.

\begin{lemma}\label{1c} For $a\in(0,1/2],\,K\in(1,\infty),\, r\in(0,1)$
and $s=\varphi^a_K(r)$, we have
\begin{enumerate}

\item  $f(r)=s^{'}\K_a(s)^2/(r^{'}\K_a(r)^2)$ is decreasing from
$(0,1)$ onto $(0,1)$\,,

\item  $g(r)=s\K^{'}_a(s)^2/(r\K^{'}_a(r)^2)$ is decreasing from
$(0,1)$ onto $(1,\infty)$\,,

\item the function $r^{'c}\K_a(r)$ is decreasing if and only if
$c\geq 2 a(1-a)$, in which case $r^{'c}\K_a(r)$ is decreasing from
$(0,1)$ onto $(0,\pi/2)$.
Moreover, $\sqrt{r^{'}}\K_a(r)$ is decreasing for all $a\in(0,1/2]$.
\end{enumerate}
\end{lemma}

\begin{lemma}\label{1d} The following formulae hold for
$a\in(0,1/2],\,r\in(0,1)$ and $x,y,K\in(0,\infty)$:

\begin{enumerate}

\item $\displaystyle\frac{d\,F}{dr}=\displaystyle\frac{l\,m}{n}F(1+l,1+m;1+n;r)
;\, F=F(l,m;n;r)$\,,

\item $\displaystyle\frac{d\K_a(r)}{dr}=\displaystyle\frac{2(1-a)
(\E_a(r)-r^{'2}\K_a(r))}{rr^{'2}}$\,,

\item $\displaystyle\frac{d\E_a(r)}{dr}=\displaystyle\frac{2(a-1)(\K_a(r)-\E_a(r))}{r}$\,,

\item $\displaystyle\frac{d\mu_a(r)}{dr}=\displaystyle\frac{-\pi^2}
{4 rr^{'2}\K_a(r)^2}$\,,

\item $\displaystyle\frac{d\varphi^{a}_K(r)}{dr}
=\displaystyle\frac{ss^{'2}\K_a(s)^2}{Krr^{'2}\K_a(r)^2}
=\displaystyle\frac{ss^{'2}\K_a(s)\K^{'}_a(s)}{rr^{'2}\K_a(r)\K^{'}_a(r)}
=K\displaystyle\frac{ss^{'2}\K^{'}_a(s)^2}{rr^{'2}\K^{'}_a(r)^2}$\,,

\item $\displaystyle\frac{d\varphi^{a}_K(r)}{dK}
=\displaystyle\frac{4ss^{'2}\K_a(s)^2\mu_a(r)}{\pi^2K^2}$\,,\\
where $s=\varphi^{a}_K(r)$,

\item $\displaystyle\frac{d\eta^{a}_K(x)}{dx}
=\frac{1}{K}\left(\frac{r^{'}s\K_a(s)}{rs^{'}\K_a(r)}\right)^2
=K\left(\frac{r^{'}s\K^{'}_a(s)}{rs^{'}\K^{'}_a(r)}\right)^2
=\left(\frac{r^{'}s}{rs^{'}}\right)^2\frac{\K_a(s)\K^{'}_a(s)}
{\K_a(r)\K^{'}_a(r)}$\,,

\item $\displaystyle\frac{d\eta^{a}_K(x)}{dK}
=\frac{8\eta^a_K(x)\mu_a(r)K_a(s)^2}{\pi^2K^2}$\,,\\
in $(7)$ and $(8)$, $r=\sqrt{x/(1+x)}$ and $s=\varphi^a_K(r)$.

\end{enumerate}
\end{lemma}

\begin{lemma}\label{thm1.52}\cite[Theorem 1.52(1)]{avvb}
For $a,b>0$, the function
$$f(x)=\frac{F(a,b;a+b;x)-1}{\log(1/(1-x))}$$
is strictly increasing from $(0,1)$ onto $(ab/(a+b),1/B(a,b))$.
\end{lemma}
\begin{subsec}{\rm {\bf Proof of Theorem \ref{nn2}.}} {\rm
With $G(r)=F(a,b;c;r^p)$, and $g$ as in Theorem \ref{nn2} we get by Lemma \ref{1d}(1)
$$g'(p)=
-\frac{\left(G(r)\right)^{1/p-1}}{c p^2} \left(c \,G(r) \log \left(\,
   G(r)\right)+a\, b\, p\, r^p\, F\left(a+1,b+1;c+1;r^p\right) \log (1/r)\right)$$
which is negative. Hence this implies (1),
and (2) follows from (1). For (3), write $F(r)=F(-a,b;c;r^p)$. We define
$h(p)=F(r)^{1/p}$ and get
$$h'(p)=
\frac{\left(F(r)\right)^{1/p-1}}{c p^2} \left(c \,F(r) \log \left(1/\,
   F(r)\right)+a\, b\, p\, r^p\, F\left(a+1,b+1;c+1;r^p\right) \log (1/r)\right)$$
which is positive because $F(r)\in(0,1)$. Hence $h$ is increasing in $p$, and (3) follows easily.
\, $\hfill  \square $  }
\end{subsec}

\begin{figure}
\includegraphics[width=12cm]{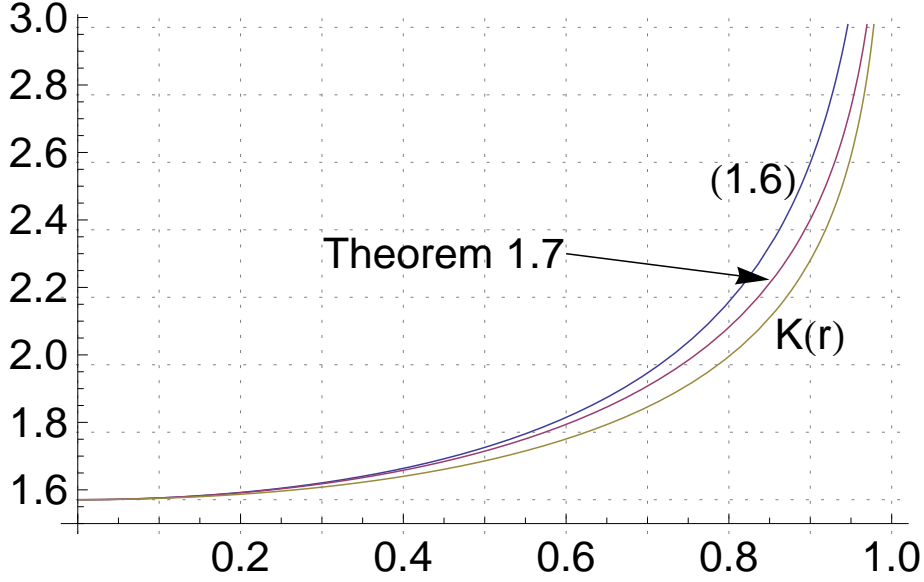}
\caption{Comparison of upper bounds given in
Theorem \ref{kk} and (\ref{aqbn}) for $\K(r)$.}
\end{figure}

\begin{subsec}{\rm {\bf Proof of Theorem \ref{kk}.}}
{\rm By the definition of ${\rm artanh_p}$, Lemma \ref{thm1.52}
and Bernoulli inequality we get
\begin{eqnarray*}
\left(\frac{{\rm artanh}_p(r)}{r}\right)^{1/2}&=&
\left(F\left(1,\frac{1}{p};1+\frac{1}{p};r^p\right)\right)^{1/2}\\
&<&\left(1-\frac{1}{p}\,\log(1- r^p)\right)^{1/2}\\
&\leq&1+\frac{1}{2p}\,\log\left(\frac{1}{1- r^p}\right)\\
&\leq&1+\frac{p-1}{p^2}\,\log\left(\frac{1}{1- r^p}\right)\\
&\leq&1-\frac{p-1}{p^2}\,\log(1- r^2)=\xi.\\
\end{eqnarray*}
Again by Lemma \ref{thm1.52} and \cite[6.1.17]{AS} we obtain
\begin{eqnarray*}
\xi&<&F\left(\frac{1}{p},1-\frac{1}{p};1;r^2\right)=\frac{2}{\pi}\K_{1/p}(r)\\
&<&1-\frac{1}{B(1/p,1-1/p)}\,\log(1- r^2)\\
&=&1-\frac{2}{p\,\pi_p}\,\log(1- r^2)\,,\\
\end{eqnarray*}
and this completes the proof. $\hfill \square$ }
\end{subsec}

\begin{subsec}{\rm {\bf Proof of Theorem \ref{kka}.}
Writing $r=1/\cosh(x)$ we have $$ \frac{dr}{dx}= - (\sinh x)/\cosh^2 x =- r\,r'  $$
and
$$
f'(x)= - \frac{\K'_a(r)}{\K_a^2(r)}\frac{dr}{dx}
=- \frac{2(1-a)}{\K_a^2(r)}  \frac{\E_a(r)- r'^2 \K_a(r)}{r\, r'^2} \,(-r\, r')
$$
 $$
 = 2(1-a)  \frac{\E_a(r)- r'^2 \K_a(r)}{r'\, \K_a(r)^2} \,,
 $$
which is positive and increasing in $r$ by Lemma \ref{1c}(3) and
therefore $f'(x)$ is decreasing in $x\,$
and $f$ is concave. Hence,
$$\frac{1}{2}(f(x)+f(y))\leq f\left(\frac{x+y}{2}\right)$$
$$\Longleftrightarrow
\frac{1}{2}\left(\frac{1}{\K_a(1/\cosh(x))}+\frac{1}{\K_a(1/\cosh(y))}\right)
\leq \frac{1}{\K_a(1/\cosh((x+y)/2))}$$
$$\Longleftrightarrow
\K_a(r)+\K_a(s)\leq
 \frac{2\K_a(r)\K_a(s)}{\K(\sqrt{rs/(1+rs+r^{'}s^{'})})}\,,$$
 using $\cosh^2((x+y)/2)=(1+rs+r^{'}s^{'})/(rs)$ and setting $s=1/\cosh(y)$.
 Clearly, $$(r-s)^2\geq 0\Longleftrightarrow1-2 rs+r^2s^2\geq 1-r^2-s^2+r^2s^2$$
 $$\Longleftrightarrow 1-rs\geq r^{'}s^{'}\Longleftrightarrow 2\geq 1+rs+r^{'}s^{'}
 \Longleftrightarrow 2rs/(1+rs+r^{'}s^{'})\geq rs\,,$$
 and the third inequality follows. Obviously, $f(0+)=0$, and $f^{'}(x)$
 is decreasing in $x$. Then $f(x)/x$ is decreasing and $f(x+y)\leq f(x)+f(y)$ by
 Lemmas \ref{pee} and \ref{pee1}, respectively. This implies the first inequality.}
\end{subsec}


\begin{subsec}{\rm {\bf Proof of Theorem \ref{mymu}.}}
{\rm By Lemma \ref{thm1.52} we get\\
(a)\quad$ 1-\displaystyle\frac{p-1}{p^2}\log r^2
<F\left(\displaystyle\frac{1}{p},1-\displaystyle\frac{1}
{p};1;1-r^2\right)<1-\displaystyle\frac{2}{p\,\pi_p}\log r^2 $\\
(b)\quad $1-\displaystyle\frac{p-1}{p^2}\log(1- r^2)
<F\left(\displaystyle\frac{1}{p},1-\displaystyle\frac{1}
{p};1;r^2\right)< 1-\displaystyle\frac{2}{p\,\pi_p}\log (1-r^2) $\,.\\

\noindent
By using (a), (b) and the definition of $\mu_a$,
we get (1).
The claim (2) is equivalent to
$$\frac{2(\pi-\log(r^2))}{4-\log(1-r^2)}<\frac{4}{\pi}\left(\frac{\pi}{2}\right)^2
\frac{4-\log(r^2)}{\pi-\log(1-r^2)}$$
$$\Longleftrightarrow 4(\pi-\log(r^2))(\pi-\log(1-r^2))
-(4-\log(r^2))(4-\log(1-r^2))<0,$$

$$ \Longleftrightarrow
(\pi-4)(4\pi-\log(r^2)\log(1-r^2))<(\pi-4)(4\pi-(\log(2))^2)<0\,.$$
For the second last inequality we define $w(x)=\log(x)\log(1-x)$, and get
$$w^{'}(x)=\frac{(1-x)\log(1-x)-x\log(x)}{x(1-x)}=\frac{-g(x)}{x(1-x)}\,,$$ and see that
$g(x)=x\log(x)-(1-x)\log(1-x)$ is convex on $(0,1/2)$ and
 concave on $(1/2,1)$. This implies that $g(x)<0$
 for $x\in(0,1/2)$ and $g(x)>0$ for $x\in(1/2,1)$.
 Therefore $w$ is increasing in $(0,1/2)$ and decreasing in $(1/2,1)$.
 Hence the function $w$ has a global maximum at $x=1/2$ and this completes the proof.

$\hfill \square$ }
\end{subsec}

One can obtain the following inequalities by using the
proof of Theorem \ref{mymu}:
$$\frac{p\,\pi_p}{2 \pi} \frac{\K_a(r)}{(1-(2/(p\,\pi_p))\log\, r^2)}
\leq\mu_a(r^{'})
\leq  \frac{p\,\pi_p}{2 \pi}  \frac{ \K_a(r)}{(1-((p-1)/p ) \log\, r^2)}\,,$$
with $a=1/p$ and $p\geq 2$.

\begin{lemma}\label{n2} The following inequalities hold for all
$r,s\in(0,1)$ and $a\in(0,1/2]$;
\begin{enumerate}
\item $\K_a(r\,s)\leq \sqrt{\K_a(r^2)\K_a(s^2)}\leq \frac{2}{\pi} \K_a(r)\K_a(s)$\,,\\
\item $\frac{2}{\pi} \E_a(r)\E_a(s)\leq
\sqrt{\E_a(r^2)\E_a(s^2)}\leq \E_a(rs)\,.$
\end{enumerate}
\end{lemma}

\begin{proof} Define $f(x)=\log(\K_a(e^{-x})),\,x>0$. We get by Lemma \ref{1d}(2)
$$f{'}(x)=-2(1-a)\frac{\E_a(r)-r^{'2}\K_a(r)}{r^{'2}\K_a(r)}\,,\quad r= e^{-x}\,,$$
and this is negative by the fact that
$h(r)=\E_a(r)-r^{'2}\K_a(r)>0$ and  decreasing in $r$
by \cite[Lemma 5.4(1)]{aqvv} and the fact that $h$ is increasing
($h{'}(r)=2 ar\K_a(r)>0$). Therefore $f{'}(x)$ is increasing in $x$,
hence $f$ is convex, and this implies the first inequality of part one.
The second inequality follows
from Theorem \ref{nn2}(2).

The first inequality of part two follows from the Theorem \ref{nn2}(3),
 for the second inequality
we define $g(x)=\log(\E_a(z)),\,z=e^{-x},\,x>0$, and get Lemma \ref{1d}(3)

$$g^{'}(x)=2(1-a)\frac{(\K_a(z)-\E_a(z))}{\E_a(z)}\,,$$
which is positive and increasing in $z$ by
\cite[Theorem 4.1(3), Lemma 5.2(3)]{aqvv}, hence $g^{'}(x)$
is decreasing in $x$, therefore $g$ is increasing and concave.
This implies that
$$\log(\E_a(e^{-(x+y)/2}))\geq (\log(\E_a(e^{-x}))+\log(\E_a(e^{-y})))/2\,,$$
and the second inequality follows if we set $r=e^{-x/2}$  and $s=e^{-y/2}$.
\end{proof}

\section{Few remarks on special functions}

In this section we generalize some results from \cite[Chapter 10]{avvb}.

\begin{theorem} The function $\mu^{-1}_a(y)$ has exactly
 one inflection point and it is log-concave
from $(0,\infty)$ onto $(0,1)$. In particular,
$$(\mu^{-1}_a(x))^{p}(\mu^{-1}_a(y))^q\leq\mu^{-1}_a(p\,x+q\,y)$$
for $p,\,q,\,x,\,y>0$ with $p+q=1$.
\end{theorem}

\begin{proof} Letting $s=\mu^{-1}_a(y)$ we see that $\mu_a(s)=y$.
By Lemma \ref{1d}(4) we get
$$\frac{ds}{dy}=-\frac{4}{\pi^2}ss^{'2}\K_a(s)^2\,,$$
\begin{eqnarray*}
\frac{d^2s}{dy^2}&=&-\frac{ds}{dy}\frac{4}{\pi^2}
(s^{'2}\K_a(s)^2-2s^{2}\K_a(s)^2+2\K_a(s)^2(\E_a(s)-s^{'2}\K_a(s)))\\
&=&\frac{16}{\pi^4}ss^{'2}\K_a(s)^3(2\E_a(s)-(1+s^2)\K_a(s))\,.
\end{eqnarray*}

We see that $2\E_a(s)-(1+s^2)\K_a(s)$ is increasing from
$(0,\infty)$ onto $(-\infty,\pi/2)$
as a function of $y$. Hence $d(\mu^{-1}_a(y_0))/dy^2=0$, for
$y_0\in(0,\infty)$ and $\mu^{-1}_a$
has exactly one inflection point. Let $f(y)=\log(\mu^{-1}_a(y))=\log\,s$. Then
$$f^{'}(y)=-\frac{4}{\pi^2} s^{'2}\K_a(s)^2\,$$
which is decreasing as a function of $y$, by Lemma \ref{1c}(3),
hence $\mu^{-1}_a$ is log-concave. This completes the
proof.
\end{proof}

\begin{corollary}\label{b} $(1)$ For $K\geq 1$, the function
$f(r)=(\log \, \varphi^a_K(r))/\log r$ is strictly decreasing
from $(0,1)$ onto $(0,1/K)$.\\
$(2)$ For $K\geq 1\,,r\in(0,1)$, the function
$g(p)=\varphi^a_K(r^p)^{1/p}$
is decreasing from
$(0,\infty)$ onto $(r^{1/K},1)$. In particular,
$$r^{p/K}\leq \varphi^a_K(r^p)\leq \varphi^a_K(r)^p,\,\,p\geq1\,,$$
and
$$\varphi^a_K(r^p)\geq \varphi^a_K(r)^p,\,\quad 0<p\leq1\,.$$
\end{corollary}

\begin{proof} Let $s=\varphi^a_K(r)$. By Lemma \ref{1d}(5) we get
$$f^{'}(r)=\frac{rss^{'2}}{srr^{'2}}
\frac{\K_a(s)\K^{'}_a(s)}{\K_a(r)\K^{'}_a(r)}
\log r-\log s,$$
and this is equivalent to
$$r(\log r)^2f^{'}(r)=s^{'2}\K_a(s)\K^{'}_a(s)
\left(\frac{\log r}{r^{'2}\K_a(r)\K^{'}_a(r)}
-\frac{\log s}{s^{'2}\K_a(s)\K^{'}_a(s)}\right),$$
which is negative by Lemma \ref{1c}(3). The limiting
 values follow from l'H\^opital Rule and
Lemma \ref{1c}(1). We observe that
$$\log g(p)=\left(\frac{\log \varphi^a_K(r^p)}
{\log(r^p)}\right)\log r\,,$$
and (2) follows from (1).
\end{proof}

\begin{lemma}\label{varph} For $0<a\leq 1/2, \,K,p\geq 1$ and $r,s\in(0,1)$,
the following inequalities hold
$$\frac{\sqrt[p]{\varphi_K^a(r^p)}+\sqrt[p]{\varphi_K^a(s^p)}}
{1+\sqrt[p]{\varphi_K^a(r^p)\varphi_K^a(s^p)}}\leq
\frac{\varphi_K^a(r)+\varphi_K^a(s)}{1+\varphi_K^a(r)
\varphi_K^a(s)}\leq
\frac{\varphi_K^a(\sqrt[p]{r})^p+\varphi_K^a(\sqrt[p]{s})^p}
{1+(\varphi_K^a(\sqrt[p]{r})\varphi_K^a(\sqrt[p]{s}))^p}\,.$$
\end{lemma}

\begin{proof} It follows from the Corollary \ref{b}(2) that
$$\varphi_K^a(r^p)^{1/p}\leq \varphi_K^a(r)\,.$$
 From the fact that
${\rm artanh}$ is increasing, we conclude that
$${\rm artanh}(\varphi_K^a(r^p)^{1/p})+{\rm artanh}(\varphi_K^a(s^p)^{1/p})\leq
{\rm artanh}(\varphi_K^a(r))+{\rm artanh}(\varphi_K^a(s))\,.$$
This is equivalent to
$${\rm artanh}\left(\frac{\varphi_K^a(r^p)^{1/p}+\varphi_K^a(s^p)^{1/p}}
{1+(\varphi_K^a(r^p)+\varphi_K^a(s^p))^{1/p}}\right)\leq
{\rm artanh}\left(\frac{\varphi_K^a(r)+\varphi_K^a(s)}
{1+(\varphi_K^a(r)+\varphi_K^a(s))}\right)\,,$$
and the first inequality holds. Similarly, the second inequality follows from
$\varphi_K^a(r)\leq \varphi_K^a(r^{1/p})^p$.
\end{proof}

For $0<a\leq 1,\,K\geq 1$ and $r,s\in(0,1)$, the following inequality
\begin{equation}\label{aq1}
\varphi_K^a\left(\frac{r+s}{1+rs}\right)\leq
\frac{\varphi_K^a(r)+\varphi_K^a(s)}{1+\varphi_K^a(r)
\varphi_K^a(s)}
\end{equation}
is given in \cite[Remark 6.17]{aqvv}.

\begin{figure}
\includegraphics[width=12cm]{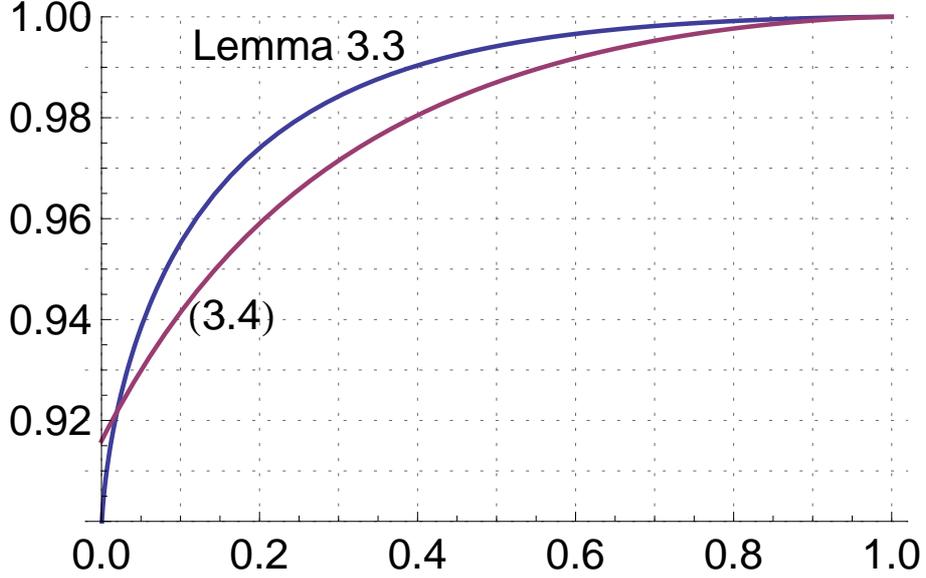}
\caption{Let $g(a,K,p,r,s)=\frac{\sqrt[p]{\varphi_K^a(r^p)}+\sqrt[p]{\varphi_K^a(s^p)}}
{1+\sqrt[p]{\varphi_K^a(r^p)\varphi_K^a(s^p)}}$,
$h(a,K,r,s)=\varphi_K^a\left(\frac{r+s}{1+rs}\right)$ be the lower bounds in Lemma
\ref{varph} and (\ref{aq1}), respectively. For $a=0.2,\,K=1.5,\,p=1.3$ and
$s=0.5$ the functions $g$ and $h$ are plotted. We see that for $r\in(0.2,1)$ the first
lower bound is better.}
\end{figure}

\begin{theorem} For $r,s\in(0,1)$, we have
\begin{enumerate}
\item $|\varphi^a_K(r)-\varphi^a_K(s)|\leq \varphi^a_K(|r-s|)
\leq e^{(1-1/K)R(a)/2}|r-s|^{1/K},\,K\geq1\,,$\\
here $R(a)$ is as in \cite[Theorem 6.7]{aqvv}
\item $|\varphi^a_K(r)-\varphi^a_K(s)|\geq \varphi^a_K(|r-s|)
\geq e^{(1-1/K)R(a)/2}|r-s|^{1/K},\,0<K\leq1\,.$\\
\end{enumerate}
\end{theorem}

\begin{proof} It follows from \cite[Theorem 6.7]{aqvv}
 that $r^{-1}\varphi^a_K(r)$
is decreasing on $(0,1)$, if $K>1$ and by Lemma \ref{pee1} we obtain
$$\varphi^a_K(x+y)\leq \varphi^a_K(x)+\varphi^a_K(y),\,\;x,y\in(0,1)\,. $$
Now the first inequality in (1) follows if we take $r=x+y$ and
$s=y$, the second one
follows from \cite[Theorem 6.7]{aqvv}. Next, (2) follows
from (1) and the fact that
$$\varphi^a_{AB}(r)=\varphi^a_{A}(\varphi^a_{B}(r)),\,\,A,B>0\,,r\in(0,1)$$
when we replace $K,\,r$ and $s$ by
$1/K,\,\varphi^a_{1/K}(r),\,\varphi^a_{1/K}(s)$,
respectively.
\end{proof}

\begin{theorem} For $a\in(0,1/2],\,c,r\in(0,1)$ and $K,L\in(0,\infty)$ we have
\begin{enumerate}
\item The function $f(K)=\log(\varphi^a_K(r))$ is
increasing and concave from
$(0,\infty)$ onto $(-\infty,0)\,.$
\item The function $g(K)={\rm artanh}(\varphi^a_K(r))$ is
increasing and convex from
$(0,\infty)$ onto $(0,\infty)\,.$
\item $\varphi^a_K(r)^c\varphi^a_L(r)^{1-c}\leq
\varphi^a_{cK+(1-c)L}(r)\leq {\rm tanh}(c\,{\rm artanh}(\varphi^a_K(r))
+(1-c)\,{\rm artanh}(\varphi^a_L(r)))\,.$
\item $$\sqrt{\varphi^a_K(r)\varphi^a_L(r)}\leq \varphi^a_{(K+L)/2}(r)$$
$$\leq \displaystyle\frac{\varphi^a_K(r)+\varphi^a_L(r)}
{1+\varphi^a_K(r)\varphi^a_L(r)
+\varphi^a_{1/K}(r^{'})\varphi^a_{1/L}(r^{'})}\,.$$
\end{enumerate}
\end{theorem}

\begin{proof} For (1), by Lemma \ref{1d}(6) we get
$$f^{'}(K)=4s^{'2}\K_a(s)^2\mu_a(r)/(\pi^2K),$$
which is positive and decreasing by Lemma \ref{1c}(3).
For (2), we get
$$f^{'}(K)=4s\K_a(s)^2\mu_a(r)/(\pi^2K^2)=s\K^{'}_a(s)^2/\mu_a(r)$$
by Lemma \ref{1d}(6), which is positive and increasing by
Lemma \ref{1c}(3).
By (1) and (2) we get
$$c\log(\varphi^a_K(r))+(1-c)\log(\varphi^a_L(r))\leq
\log(\varphi^a_{cK+(1-a)L}(r)),$$
$${\rm artanh}(\varphi^a_{cK+(1-c)K}(r))\leq a\,{\rm artanh}
(\varphi^a_K(r))+(1-c)\,{\rm artanh}(\varphi^a_L(r)),$$
respectively, and (3) follows. Also
$$(\log(\varphi^a_K(r))+\log(\varphi^a_L(r)))/2\leq
\log(\varphi^a_{(K+L)/2}(r))\,, $$
and
$${\rm artanh}(\varphi^a_{(K+L)/2}(r))\leq ({\rm artanh}
(\varphi^a_K(r))+{\rm artanh}(\varphi^a_L))/2\,,$$
follow from (1) and (2), and hence (4) holds.
\end{proof}

\begin{theorem}\label{a} For $K\geq 1$ and $0<m<n$,
the following inequalities hold
\begin{enumerate}
\item $\displaystyle\eta^a_K(m\,n)\leq \displaystyle
\sqrt{\eta^a_K(m^2)\eta^a_K(n^2)}\,,$
\item $\displaystyle\left(\frac{n}{m}\right)^{1/K}<
\displaystyle\frac{\eta^a_K(n)}{\eta^a_K(m)}<
\left(\frac{n}{m}\right)^K\,,$
\item $\eta^a_K(m)\eta^a_K(n)<\displaystyle\left(\eta^a_K
\displaystyle\left(\frac{m+n}{2}\right)\right)^2\,,$
\item $2\displaystyle\frac{\eta^{a}_K(m)\eta^{a}_K(n)}
{\eta^{a}_K(m)+\eta^{a}_K(n)}<\eta^{a}_K(\sqrt{m\,n})<
\sqrt{\eta^{a}_K(m)\eta^{a}_K(n)}$.
\end{enumerate}
\end{theorem}

\begin{proof} We define a function $g(x)=\log\eta^a_K(e^x)$
on $\mathbb{R}$.
By [AQVV, Theorem 1.16], $g$
is increasing, convex and satisfies $1/K\leq g^{'}(x)\leq K$. Then
\begin{eqnarray*}
\log\eta^a_K(e^{(x+y)/2})&=&g\left(\frac{x+y}{2}\right)\leq
\frac{g(x)+g(y)}{2}\\
                         &=& \frac{1}{2}\log(\eta^a_K(e^x))
                         +\frac{1}{2}\log(\eta^a_K(e^y)),
\end{eqnarray*}
and this is equivalent to
$$\log\eta^a_K(e^{x/2}e^{y/2}) \leq \log
(\eta^a_K(e^{x/2})\eta^a_K(e^{y/2}))\,.$$
Hence (1) follows if we set $e^{x/2}=m$ and $e^{y/2}=n$.
For (2), let $x>y$. Then by the inequality $1/K\leq g^{'}(x)\leq K$
 and the mean value theorem we get
$$(x-y)/K\leq g(x)-g(y)\leq K(x-y),$$
and this is equivalent to
$$(\log(e^x)-\log(e^y))/K\leq \log(\eta^a_K(e^x))-\log(\eta^a_K(e^y))
\leq K(\log(e^x)-\log(e^y))\,.$$
By setting $e^{x/2}=m$ and $e^{y/2}=n$ we get the desired inequality.
For (3), let $f(x)=\log(\eta^a_K(x)),\,r=\sqrt{x/(1+x)}$
and $s=\varphi^a_K(r)$.
Then by Lemma \ref{1d}(7) we get
\begin{eqnarray*}
f^{'}(x)&=&\frac{1}{K}\left(\frac{s^{'}}{s}\right)^2
\left(\frac{sr^{'}\K_a(s)}{rs^{'}\K_a(r)}\right)^2=
         \frac{1}{K}\left(\frac{r^{'}}{r}\right)^2
         \left(\frac{\K_a(s)}{\K_a(r)}\right)^2\\
         &=&\frac{1}{K}\left(\frac{r^{'}}{s}\right)^2
         \left(\frac{s\K_a(s)}{r\K_a(r)}\right)^2,
\end{eqnarray*}
which is positive and decreasing by Lemma \ref{1c}(2).
Hence $(f(x)+f(y))/2\leq f((x+y)/2)$, and the
inequality follows.\\
For (4), letting $h(x)=1/\eta^a_K(e^x)$, we see that this is log-concave by (1),
and we get

$$\frac{\log(1/\eta^a_K(e^x))+\log(1/\eta^a_K(e^y))}{2}<
\log(1/\eta^a_K(e^{(x+y)/2}))\,,$$
Setting $e^x=m$ and $e^y=n$ we get the second inequality. We observe that
$h(x)=(s^{'}/s),\, s=\varphi^a_K(r),\, r=\sqrt{e^x/(e^x+1)}$. We get
$$-f^{'}(x)=\frac{1}{K}\left(\frac{r^{'}}{s}\right)
\left(\frac{s^{'}\K_a(s)}{r^{'}\K_a(r)}\right)^2,$$
which is positive and decreasing by Lemma \ref{1c}(1),
hence $h$ is convex, and the first inequality follows easily.
\end{proof}

\begin{theorem}\label{c} For $x\in(0,\infty)$,
the function $f:(0,\infty)\to (0,\infty)$
defined by $f(K)=\eta^a_K(x)$ is increasing,
convex and log-concave. In particular,
$$\eta^a_K(x)^{c}\eta^a_L(x)^{1-c}\leq \eta^a_{cK+(1-c)L}(x)\leq c\,\eta^a_K(x)+(1-c)\eta^a_L(x)$$
for $K,L,x\in(0,\infty)$ and $c\in(0,1)$,
with equality if and only if $K=L$.
\end{theorem}

\begin{proof} We observe that $f(K)=(s/s^{'})^2$, where
$s=\varphi^a_K(r)$ and $r=\sqrt{x/(x+1)}$. We get by
Lemma \ref{1d}(8)
$$f^{'}(K)=\frac{8s^2\K_a(s)^2}{\pi^2 s^{'2}K^2}\mu_a(r)=
\frac{4}{\pi \sin(\pi a)}\frac{\K_a(r)}{\K^{'}_a(r)}
\left(\frac{s\K^{'}_a(s)}{s^{'}}\right)^2,$$
which is positive and increasing by Lemma \ref{1c}(3),
 hence $f$ is increasing and convex. For log-concavity,
let $g(K)=\log(\eta^a_K(x))$. By Lemma \ref{1d}(8) we get
$$g^{'}(K)=\frac{8\K_a(s)^2}{\pi^2K^2}\mu_a(r)=\frac{4}{\pi \sin(\pi a)}\frac{\K_a(r)}
{\K^{'}_a(r)}\K^{'}_a(s)^2,$$
which is decreasing, hence $f$ is log-concave.
\end{proof}

\begin{theorem}\label{d} The function
$$f(K)=\frac{\log\eta^a_K(x)-\log(x)}{K-1}$$
is decreasing from $(1,\infty)$ onto
 $$\left(\frac{\pi\K_a(r)}{\sin(\pi\,a)\K^{'}_a(r)},
\frac{4\K_a(r)\K^{'}_a(r)}{\pi\sin(\pi\,a)}\right),$$ and
the function
$$g(K)=\frac{\eta^a_K(x)-(x)}{K-1}$$
is increasing from $(1,\infty)$ onto
$$(4r^2\sin(\pi\,a)\K_a(r)\K^{'}_a(r)/(\pi r^{'2}),\infty),$$
where $r=\sqrt{x/(x+1)}$.
\end{theorem}

\begin{proof} It follows from Theorem \ref{c} and
Lemma \ref{pee}
that $f$ is monotone. Let $s=\varphi^a_K(r)$,
by Lemma \ref{1d}(6), the l'H\^opital
Rule and definition of $\mu_a$ we get
$$\lim_{K\to 1}f(K)=\lim_{K\to 1}\frac{2}{K-1}
\log\left(\frac{sr^{'}}{s^{'}r}\right)$$
$$=\lim_{K\to 1}\frac{8\K_a(s)^2\mu_a(r)}{K^2\pi^2}=
\frac{8}{\pi^2}\K_a(r)^2\mu_a(r)=
\frac{4\K_a(r)\K^{'}_a(r)}{\pi\sin(\pi\,a)}.$$
By using the fact that $K=\mu_a(r)/\mu_a(s)$ and
the l'H\^opital Rule, we get
$$\lim_{K\to \infty}f(K)=\lim_{K\to \infty}
\frac{8\mu_a(s)^2\K_a(s)^2}{\pi^2\mu_a(r)}$$
$$=\lim_{K\to \infty}\frac{2\K^{'}_a(s)^2}{\sin^2(\pi\,a)\mu_a(r)}=
\frac{2\K_a(0)^2}{\sin^2(\pi\,a)\mu_a(r)}
=\frac{\pi\K_a(r)}{\sin(\pi\,a)\K^{'}_a(r)}.$$
Next, let $g(K)=G(K)/H(K)$, where $G(K)=(s/s^{'})^2-(r/r^{'})^2$
and $H(K)=K-1$. We see that
$G(1)=H(1)=0$ and $G(\infty)=H(\infty)=\infty$.
We see that $$G^{'}(K)/H^{'}(K)=
2(s\K^{'}_a(s))^2/(s^{'2}\mu_a(r))\,,$$
and it follows from Lemma \ref{1c}(3)
and Lemma \ref{pee} that $g(K)$ is increasing and the required
limiting values follow from
$\varphi^a_K(r)=\mu^{-1}_a(\mu_a(r)/K)$.
\end{proof}

\begin{remark}\label{e} \rm If we take $x=1$ in
Theorem \ref{d}, then with $t=4\K_a(1/\sqrt{2})^2/(\pi\sin(\pi\,a))$ we have
\begin{enumerate}
\item the function $\log(\lambda_a(K))/(K-1)$
is strictly decreasing from $(1,\infty)$ onto
$(\pi/\sin(\pi\,a),t)$, and
\item the function $(\lambda_a(K)-1)/(K-1)$ is increasing from
$(1,\infty)$ onto $(t\,\sin^2(\pi\,a),\infty)$.
\end{enumerate}
In particular,
$$e^{\pi(K-1)/\sin(\pi\,a)}<\lambda_a(K)<e^{t(K-1)},$$

$$1+t(K-1)\sin^2(\pi\,a)<\lambda_a(K)<\infty,$$
respectively, and we get
$$\max\{e^{\pi(K-1)/\sin(\pi\,a)},1+t(K-1)\sin^2(\pi\,a)\}<\lambda_a(K)<e^{t(K-1)}.$$
\end{remark}

\begin{lemma}\label{f0} For $c\in[-3,0)$, the function
$f(r)=\K_a(r)^c+\K^{'}_a(r)^c$
is strictly increasing from $(0,1/\sqrt{2})$ onto
$((\pi/2)^c,2\K_a(1/\sqrt{2})^c)$.
\end{lemma}

\begin{proof} By Lemma \ref{1d}(2) we get
$$f^{'}(r)=\frac{2 (1-a) c \K_a(r)^{c-1} \left(\E_a(r)-r^{'2}
 \K_a(r)\right)}{r r^{'}}-\frac{2 (1-a) c \K^{'}_a(r)^{c-1}
 (\E^{'}_a(r)-r^2\K^{'}_a(r))}{r r^{'}}$$
   $$=\frac{2 (1-a) c (\K_a(r)\K^{'}_a(r))^{c-1} }
   {r r^{'}}(h(r)-h(r^{'})),$$
and here $h(r)=\displaystyle\frac{r^2 \K^{'}_a(r)^{1-c}}{r^2} (\E_a(r)-r^{'2}
 \K_a(r)),$ which is increasing on $(0,1)$
 by \cite[Theorem 3.21(1)]{avvb} and Lemma \ref{1c}(3).
Hence $f^{'}(r)<0$ on $(0,1/\sqrt{2})$,
 and the limiting values are clear.
\end{proof}

\begin{theorem}\label{f} $(1)$ For $K>1$, the function
$(\log(\lambda_a(K))/(K-1/K)$
is strictly increasing from $(1,\infty)$ onto
$(2\K_a(1/\sqrt{2})/(\pi\,\sin(\pi\,a)),\pi/\sin(\pi\,a)
)$.\\
$(2)$ The function $\log(\lambda_a(K)+1)$ is convex on
$(0,\infty)$, and $\log(\lambda_a(K))$ is concave.\\
$(3)$ The function $g(K)=(\log(\lambda_a(K)))/\log K$ is
 strictly increasing on $(1,\infty)$. In particular, for
$c\in(0,1)$
$$\lambda_a(K^c)<(\lambda_a(K))^c\,.$$
\end{theorem}

\begin{proof} For (1), let $r=\mu^{-1}_a(\pi K/(2\sin(\pi\,a))),\,
0\leq r \leq 1/\sqrt{2}$. Then by (\ref{1aa})
$$r^{'}=\sqrt{1-\left(\mu^{-1}_a\left(\frac{\pi K}
{2\sin(\pi\,a)}\right)\right)^2}$$
$$=\sqrt{1-\left(\mu^{-1}_a
\left(K\mu_a\left(\frac{1}{\sqrt{2}}\right)\right)\right)^2}
=\mu^{-1}_a\left(\frac{\pi}{2K\sin(\pi\,a)}\right),$$
we also observe that $K=\K^{'}_a(r)/\K_a(r)$.
Now it is enough to prove that the function

$$f(r)=\frac{2\log(r^{'}/r)}{\K^{'}_a(r)/\K_a(r)-\K_a(r)\K^{'}_a(r)}
=\frac{\pi\log(r^{'}/r)}{\sin(\pi\,a)(\mu_a(r)+\mu_a(r^{'}))},$$
is strictly decreasing on $(0,1/\sqrt{2})$. Set $f(r)=G(r)/H(r)$. Clearly,
$G(1/\sqrt{2})=H(1/\sqrt{2})=0$. By Lemma \ref{1d}(4) we get
$$\frac{G^{'}(K)}{H^{'}(K)}=\frac{4}{\pi\sin(\pi\,a)
(\K_a(r)^{-2}-\K_a(r^{'})^{-2})},$$
which is strictly decreasing from $(0,1/\sqrt{2})$ onto
$$(2\K_a(1/\sqrt{2})/(\pi\,\sin(\pi\,a)),\pi/\sin(\pi\,a))$$
 by Lemma \ref{f0}.
Now the proof of (1) follows from Lemma \ref{pee}.
For (2), it follows from Theorem \ref{c} that
$\log (\lambda_a(K))$ is concave. Letting
$f(K)=\lambda_a(K)+1$ we have
$$f(K)=\left(\mu^{-1}_a\left(\frac{\pi K}
{2\sin(\pi\,a)}\right)\right)^{-2},$$
by (\ref{1b}) and (\ref{1aa}).
Now we have $\log f(K)=-2\log y$, here
$\mu_a(y)=\pi K /(2\sin(\pi \,a))$.
By Lemma \ref{1d}(4) we get
$$\frac{f^{'}(K)}{f(K)}=-\frac{2}{y}\frac{dy}{dK}
=\frac{4}{\pi}(y^{'}\K_a(y)),$$
which is decreasing in $y$ by Lemma \ref{1c}(3),
and increasing in $K$. Hence
$\log f(K)$ is convex.

For (3), $K>1$, let $h(K)=(K-1/K)/\log K$. We get
$$h^{'}(K)=\frac{(1+K^2)\log K-(K^2-1)}{(K\log K)^2},$$
which is positive because
$$\log K>\frac{2(K-1)}{K+1}>\frac{K^2-1}{K^2+1}$$
by \cite[1.58(4)a]{avvb}, hence $h$ is strictly increasing. Also
$$g(K)=h(K)\frac{\log(\lambda_a(K))}{K-1/K}=
\frac{\log(\lambda_a(K))}{\log K}$$
is strictly increasing by (1). This implies that
$$\frac{\log (\lambda_a(K^c))}{c\log K}<\frac{\log
(\lambda_a(K))}{\log K},$$
and hence (3) follows.
\end{proof}

\begin{corollary}\label{g} For $ 0<r<1/\sqrt{2}$ and
$t=\pi^2/(2\K_a(1/\sqrt{2})^2)$, we have \\
$(1)$ The function $f(r)=(\mu_a(r)-\mu_a(r^{'}))/\log(r^{'}/r)$
is increasing from $(0,1/\sqrt{2})$ onto $(1,t)$. In particular,
$$\log(r^{'}/r)<\mu_a(r)-\mu_a(r^{'})<\frac{\pi^2}
{2\K_a(1/\sqrt{2})^2}\log(r^{'}/r).$$
$(2)$ For $g(r)=\log(r^{'}/r)$,
$$g(r)+\sqrt{(\pi/\sin(\pi\,a))^2+g(r)^2}<2\mu_a(r)<t\,g(r)+
\sqrt{(\pi/\sin(\pi\,a))^2+t^2\,g(r)^2}.$$
\end{corollary}

\begin{proof} It follows from the proof of Theorem
\ref{f}(1) that $f(r)$ is increasing,
and limiting values follows easily by the l'H\^opital Rule.
For (2), from the definition of $\mu_a$ we get
$\mu_a(r^{'})=\pi^2/((2\sin(\pi\,a))^2\mu_a(r))$,
replacing this in (1) we obtain
$$1<\frac{\mu_a(r)^2-\pi^2/(2\sin(\pi\,a))^2}
{\mu_a(r)\log(r^{'}/r)}<t=
\frac{\pi^2}{2\K_a(1\sqrt{2})^2}\,.$$
This implies that
\begin{equation}\label{1i}
\mu_a(r)^2-\mu_a(r)\log(r^{'}/r)>\frac{\pi^2}{(2\sin(\pi\,a))^2}
\end{equation}
and
\begin{equation}\label{1j}
\mu_a(r)^2-t\,\mu_a(r)\log(r^{'}/r)<\frac{\pi^2}{(2\sin(\pi\,a))^2}\,.
\end{equation}
We get the left and right inequalities by solving (\ref{1i})
and (\ref{1j}) for $\mu_a(r)$, respectively.
\end{proof}

\section{Three-parameter complete elliptic integrals}
The results in this section have counterpart in \cite{aqvv}.
For $a,b,c>0,\,a+b\geq c$, the decreasing homeomorphism $\mu_{a,b,c}:(0,1)\to(0,1)$, defined by
$$\mu_{a,b,c}(r)=\frac{B(a,b)}{2}\frac{F(a,b;c;r^{'2})}{F(a,b;c;r^2)},\,r\in(0,1)$$
where $B$ is the beta function. The $(a,b,c)$-modular function is defined by
$$\varphi_K^{a,b,c}(r)=\mu_{a,b,c}^{-1}(\mu_{a,b,c}(r)/K)\,.$$
We denote, in case $a<c$
$$\mu_{a,c}(r)=\mu_{a,c-a,c}(r)\quad{\rm and}\quad \varphi_K^{a,c}(r)=\varphi_K^{a,c-a,c}(r)\,.$$
We define the three-parameter complete elliptic integrals of the first and second kinds for
$0<a<\min\{c,1\}$ and $0<b<c\leq a+b$, by
$$\K_{a,b,c}(r)=\frac{B(a,b)}{2}F(a,b;c;r^{2})$$
$$\E_{a,b,c}(r)=\frac{B(a,b)}{2}F(a-1,b;c;r^{2})\,,$$
and denote
$$\K_{a,c}(r)=\K_{a,c-a,c}(r)\quad{\rm and}\quad \E_{a,c}(r)=\E_{a,c-a,c}(r)\,.$$

\begin{lemma}\label{pc}\cite[Theorem 3.6]{hlvv} For $0<a<c\leq1$,
the function
$f(r)=\mu_{a,c}(r){\rm artanh}\,r$ is strictly increasing from
 $(0,1)$ onto $(0,(B/2)^2)$.
\end{lemma}

\begin{lemma}\label{pd}\cite[Lemma 4.1]{hlvv} Let
$a<c\leq 1,\,K\in(1,\infty),
\,r\in(0,1)$, and let $s=\varphi^{a,c}_{K}(r)$ and
$t=\varphi^{a,c}_{1/K}(r)$. Then the function
\begin{enumerate}
\item $f_1(r)=\K_{a,c}(s)/\K_{a,c}(r)$ is increasing
from $(0,1)$ onto $(1,K)$,
\item $f_2(r)=s^{'}\K_{a,c}(s)^2/(r^{'}\K_{a,c}(r)^2)$
is decreasing from $(0,1)$
 onto $(0,1)$,
\item $f_3(r)=s\K^{'}_{a,c}(s)^2/(r\K^{'}_{a,c}(r)^2)$
is decreasing from $(0,1)$ onto $(1,\infty)$,
\item $g_1(r)=\K_{a,c}(t)/\K_{a,c}(r)$ is decreasing from
$(0,1)$ onto $(1/K,1)$,
\item $g_2(r)=t^{'}\K_{a,c}(t)^2/(r^{'}\K_{a,c}(r)^2)$ is
 increasing from $(0,1)$
 onto $(1,\infty)$,
\item $g_3(r)=t\K^{'}_{a,c}(t)^2/(r\K^{'}_{a,c}(r)^2)$ is
 increasing from $(0,1)$
onto $(0,1)$,
\item $g_4(r)=s/r$ is decreasing from $(0,1)$ onto $(1,\infty)$,
\item $g_5(r)=t/r$ is increasing from $(0,1)$ onto $(0,1)$.
\end{enumerate}
\end{lemma}

\begin{theorem}\label{muac} For $0<a<c\leq 1$, the
function $f(x)=\mu_{a,c}(1/\cosh(x))$
is increasing and concave from $(0,\infty)$ onto
$(0,\infty)$. In particular,
$$\mu_{a,c}\left(\frac{rs}{1+r^{'}s^{'}}\right)\leq \mu_{a,c}(r)
+\mu_{a,c}(s)\leq
2\mu_{a,c}\left(\sqrt{\frac{2rs}{1+rs+r^{'}s^{'}}}\right)\,,$$
for all $r,s\in(0,1)$. The second inequality becomes equality
if and only if $r=s$.
\end{theorem}

\begin{proof} Let $r=1/\cosh(x)$ and (cf. \cite{hlvv})
$$M(r^2) = \left(\frac{2}{B(a,b)}\right)^2  \, b
\, ( \K_{a,c}(r)\E'_{a,c}(r)+\K'_{a,c}(r)\E_{a,c}(r) -\K_{a,c}(r)\K'_{a,c}(r))\,.$$
We get
$$f^{'}(x)=\frac{B(a,b)}{2}\frac{M(r^2)}{r^{'2}\K(r)^2}\,,$$
which is positive and increasing in $r$ by
\cite[Lemma 3.4(1), Theorem 3.12(2)]{hlvv}, and $f$
is decreasing in $x$. Hence $f$ is concave.
This implies that
$$\frac{1}{2}\left(\mu_{a,c}\left(\frac{1}{\cosh(x)}\right)
+\mu_{a,c}\left(\frac{1}{\cosh(y)}\right)\right)
\leq \mu_{a,c}\left(\frac{1}{\cosh((x+y)/2)}\right)\,,$$
and we get the second inequality by using the formula
$$\left(\cosh\left(\frac{x+y}{2}\right)\right)^2=\frac{1+rs+r^{'}s^{'}}{2rs}$$
and setting $s=1/\cosh(y)$. Next, $f^{'}(x)$
is decreasing in $x$, and $f(0)=0$. Then
$f(x)/x$ is decreasing on $(0,\infty)$ and
$f(x+y)\leq f(x)+f(y)$ by Lemmas \ref{pee}
and \ref{pee1}, respectively.
Hence the first inequality follows.
\end{proof}

\begin{lemma} For $0<a<c\leq 1$, we have
$$\mu_{a,c}(r)+\mu_{a,c}(s)\leq 2\mu_{a,c}(\sqrt{rs})\,,$$
for all $r,s\in(0,1)$, with equality if and only if $r=s$.
\end{lemma}
\begin{proof} Clearly,
$$(r-s)^2\geq 0 \Longleftrightarrow 1+r^2s^2\geq
1-(r-s)^2+r^2s^2 $$
$$ \Longleftrightarrow (1-rs)^2\geq 1-r^2-s^2+r^2s^2
\Longleftrightarrow 1-rs\geq r^{'}s^{'}$$
$$\Longleftrightarrow 2\geq 1+rs+r^{'}s^{'}\Longleftrightarrow
1/(rs)\geq (1+rs+r^{'}s{'})/(2rs)\,.$$
By using the fact that $\mu_{a,c}$ is decreasing, we get
$$\mu_{a,c}\left(\sqrt{\frac{2rs}{1+rs+r^{'}s^{'}}}\right)\leq \mu_{a,c}(\sqrt{r\,s})\,,$$
and the result follows from Theorem \ref{muac}.
\end{proof}

\begin{theorem}\label{p1C} For $K>1$, $0<a<c$ and $r,s\in(0,1)$,
$${\rm tanh}(K {\rm artanh}\,r)< \varphi^{a,c}_K(r).$$
The inequality is reversed if we replace $K$ by $1/K$.
\end{theorem}

\begin{proof} Let $s=\varphi^{a,c}_K(r)$. Then $s>r$,
 and by equality
$\varphi^{a,c}_K(r)=\mu^{-1}_{a,c}(\mu_{a,c}(r)/K)$ and
 Lemma \ref{pc} we get
$$\frac{1}{K}\mu_{a,c}(r){\rm artanh}\,s=\mu_{a,c}(s){\rm artanh}\,s
>\mu_{a,c}(r){\rm artanh}\,r,$$
which is equivalent to the required inequality. For the case $1/K$ let
 $x=\varphi^{a,c}_{1/K}(r)$. Then
$x<r$, and similarly we get
$$K\mu_{a,c}(r){\rm artanh}\,x=\mu_{a,c}(x){\rm artanh}\,x
<\mu_{a,c}(r){\rm artanh}\,r\,,$$
and this is equivalent to ${\rm tanh}( ({\rm artanh}\,r)/K)> \varphi^{a,c}_{1/K}(r).$
\end{proof}

\end{document}